\documentclass[11pt,final]{article}
\usepackage{amsmath,amsthm,amsfonts,amssymb,graphicx,hyperref,enumerate,psfrag}
\usepackage{fullpage}

\usepackage[utf8]{inputenc} 

\newtheorem{lemma}{Lemma}

\newtheorem{remark}{Remark}

\newtheorem{theorem}{Theorem}

\newtheorem{corollary}{Corollary}

\newcommand{\vertiii}[1]{{\left\vert\kern-0.25ex\left\vert\kern-0.25ex\left\vert #1 
    \right\vert\kern-0.25ex\right\vert\kern-0.25ex\right\vert}}
\newcommand{\EE}{{\mathbb{E}}}

\newcommand{\R}{\mathbb {R}}

\newcommand{\cI}{\mathcal {I}}

\newcommand{\cE}{\mathcal {E}}

\newcommand{\cR}{\mathcal {R}}
\newcommand{\cL}{\mathcal {L}}
\newcommand{\cP}{\mathcal {P}}

\newcommand{\tmix}{t_{\textsc{mix}}}

\newcommand{\ent}{{\mathrm{Ent}}}

\newcommand{\opi}{\widehat{\pi}}
\newcommand{\oQ}{\widehat{Q}}
\newcommand{\of}{\widehat{f}}
\newcommand{\ok}{\widehat{\kappa}}

\newcommand{\e}{{\mathfrak{e}}}
\newcommand{\bm}{{\underline{{\mathfrak m}}}}
\newcommand{\bM}{\overline{\mathfrak m}}
\newcommand{\Var}{{\mathrm{Var}}}
\newcommand\preceqdot{\mathrel{\ooalign{$\preceq$\cr
  \hidewidth\raise0.225ex\hbox{$\cdot\mkern0.5mu$}\cr}}}

\title{Modified log-Sobolev inequalities for strong-Rayleigh measures}
\author{Jonathan Hermon, Justin Salez}

\usepackage{color}

\begin{document}

\maketitle
\begin{abstract}
We establish universal modified log-Sobolev inequalities for reversible Markov chains on the boolean lattice $\{0,1\}^n$, under the only assumption that the invariant law $\pi$ satisfies a form of negative dependence  known as the \emph{stochastic covering property}. This condition is strictly weaker than the  strong Rayleigh property, and is satisfied in particular by all determinantal measures, as well as the uniform distribution over the set of bases of any balanced matroid. In the special case where $\pi$ is $k-$homogeneous, our results imply the celebrated concentration inequality for Lipschitz functions due to Pemantle \& Peres (2014). As another application, we deduce that the natural Monte-Carlo Markov Chain used to sample from $\pi$  has mixing time at most $kn\log\log\frac{1}{\pi(x)}$ when initialized in state $x$. To the best of our knowledge, this is the first work relating negative dependence and modified log-Sobolev inequalities.
\end{abstract}

\tableofcontents

\section{Introduction}
\subsection{Functional inequalities for Markov chains}

Consider a reversible Markov generator $Q$ with respect to some probability distribution $\pi$ on a finite state space  $\Omega$. In other words, $Q$ is a $\Omega\times\Omega$ matrix with non-negative off-diagonal entries, with each row summing up to $0$, and satisfying the \emph{local balance} equations
\begin{eqnarray}
\label{reversible}
\pi(x)Q(x,y) & = & \pi(y)Q(y,x),\qquad \qquad (x,y)\in \Omega^2.
\end{eqnarray}
When $Q$ is irreducible, the underlying Markov semi-group $p_t=e^{tQ}$ \emph{mixes}, in the sense that
\begin{eqnarray}
\label{mixing}
p_t(x,y) & \xrightarrow[t\to\infty]{} & \pi(y),
 \end{eqnarray} 
for all states $x,y\in\Omega$. The time-scale on which this convergence occurs is traditionally measured by the so-called \emph{mixing-times}, defined for any initial state $x\in\Omega$ and any precision $\varepsilon\in(0,1)$ by
\begin{eqnarray}
\label{def:tmix}\tmix(x;\varepsilon) & := & \min\left\{t\ge 0\colon \left\|p_t(x,\cdot)-\pi\right\|_{\textsc{tv}}\le \varepsilon\right\},
\end{eqnarray}
where $\|\mu-\nu\|_\textsc{tv}=\max_{A\subseteq \Omega}\left|\mu(A)-\nu(A)\right|$. 
Estimating this fundamental parameter is a fascinating theoretical problem with many practical applications, see e.g. the books \cite{MR2341319,MR3726904}. 

Throughout the paper, we will use the standard notation $\EE_\pi\left[\cdot\right]$, $\Var_\pi\left(\cdot\right)$ and $\ent_\pi(\cdot)$ for expectation, variance and entropy over the probability space $(\Omega,\cP(\Omega),\pi)$, i.e.
\begin{eqnarray}
\EE_\pi\left[f\right] & := & \sum_{x\in\Omega}\pi(x)f(x),\\
\Var_\pi(f) & := & \EE_\pi\left[f^2\right]-\EE_\pi^2[f],\\
\ent_\pi(f) & := & \EE_\pi\left[f \log f\right]-\EE_\pi[f]\log\EE_\pi[f].
\end{eqnarray}
All logarithms appearing here are {natural} logarithms. Some of the most powerful controls on  mixing-times are obtained by establishing appropriate {functional inequalities}  for the \emph{Dirichlet form}:
\begin{eqnarray}
\cE_\pi(f,g) & := & -\EE_\pi\left[fQg\right].
\end{eqnarray}
In particular, the \emph{Poincaré constant} $\lambda(Q)$, the \emph{modified log-Sobolev constant} $\alpha(Q)$, and the  \emph{log-Sobolev constant} $\rho(Q)$  are respectively defined as the largest numbers $\lambda,\alpha,\rho\ge 0$ such that 
\begin{eqnarray}
\label{PI}
 \cE_\pi(f,f) & \ge & \lambda\Var_\pi(f);\\
\label{MLSI}
\cE_\pi\left(f,\log f\right) & \ge & \alpha\,\ent_\pi(f);\\
\label{LSI}
\cE_\pi\left(\sqrt{f},\sqrt{f}\right) & \ge & \rho\,\ent_\pi(f).
\end{eqnarray}
for all positive observables $f\colon \Omega\to(0,\infty)$.
It is classical that these inequalities are increasing in strength, in the precise sense that one always has
 \begin{eqnarray}
\label{hier}
2\lambda(Q) & \ge \ \alpha(Q) \  \ge & 4\rho(Q). 
 \end{eqnarray} 

The constants $\lambda(Q),\alpha(Q),\rho(Q)$ provide fundamental controls on the  underlying Markov semi-group, from transport and hypercontractivity to quantitative rates of convergence with respect to stronger metrics than total-variation distance. We refer the reader to the survey papers \cite{MR1410112,MR2283379} for details. Let us here only mention the following classical implications in terms of mixing-times:
\begin{eqnarray}
\label{mixing:PI}
\tmix (x;\varepsilon) & \le & \frac{1}{2\lambda(Q)}\left(\log\frac{1}{\pi(x)}+\log\frac{1}{4\varepsilon^2}\right),\\
\label{mixing:MLSI}
\tmix (x;\varepsilon) & \le & \frac{1}{\alpha(Q)}\left(\log\log\frac{1}{\pi(x)}+\log\frac{1}{2\varepsilon^2}\right),\\
\tmix (x;\varepsilon) & \le & \frac{1}{4\rho(Q)}\left(\log\log\frac{1}{\pi(x)}+\log\frac{1}{2\varepsilon^2}\right).
\end{eqnarray}
Another important motivation for establishing functional inequalities is the \emph{concentration-of-measure} phenomenon \cite{MR3185193,MR1767995}. For example, a now-classical use of the \emph{Herbst argument} (see Section \ref{sec:herbst}) yields the following sub-Gaussian tail estimate: for all $f\colon\Omega\to\R$ and $a\ge 0$,
\begin{eqnarray}
\label{conc:MLSI}
\pi\left(f \ge  \EE_\pi(f)+a\right) & \le & \exp\left(-\frac{\alpha(Q)a^2}{4v(f)}\right),
\end{eqnarray}
 where $v(f)$ denotes the maximal  one-sided  quadratic variation of $f$ under $Q$, i.e.
\begin{eqnarray}
v(f) & := & \max_{x\in\Omega}\left\{\sum_{y\in\Omega}Q(x,y)\left[f(y)-f(x)\right]^2_+\right\}.
\end{eqnarray}

In view of the above estimates, it is clear that the modified log-Sobolev inequality $(\textbf{MLSI})$ (\ref{MLSI}) is much more powerful than a Poincaré inequality (\ref{PI}) with the same constant. Unfortunately, establishing sharp lower bounds on the modified log-Sobolev constant of practical-purpose Markov chains remains a notoriously challenging task. The aim of the present paper is to provide simple, universal functional-analytic estimates in the case where the state space $\Omega$ is a subset of the boolean lattice $\{0,1\}^n$, and the probability law $\pi$ satisfies an appropriate form of negative dependence known as the \emph{stochastic covering property} (\textbf{SCP}) \cite{MR3197973}. The precise definition will be recalled in the next section, but let us already point out that the \textbf{SCP}  is strictly weaker than the well-known strong Rayleigh property, and that it is satisfied in particular by all determinantal measures, as well as the uniform distribution over the set of bases of any balanced matroid. 

To the best of our knowledge, this is the first work relating negative dependence and modified log-Sobolev inequalities. Among other consequences, our results imply the celebrated sub-Gaussian concentration estimate for Lipschitz functions due to Pemantle \& Peres (2014) \cite{MR3197973}. In addition, we develop a new general framework for proving discrete functional inequalities using couplings, which should be applicable to many other settings. In particular, since the first appearance of our paper online, we have used it in  \cite{ZRPMLSI} to confirm a prediction of Caputo and Posta regarding the validity of the dimension-free \textbf{MLSI} for the mean-field zero-range process \cite{ZRPconjecture} with increasing rates. Also, a matrix version of our strategy was recently developed in \cite{matrix} to establish matrix Poincar\'e inequalities. Finally, in a work in preparation, we employ our method to prove a log-Sobolev inequality for generalized exclusion processes.

\subsection{Negative dependence for binary variables} Among the plethora of possible notions of negative dependence for binary  variables $X_1,\ldots,X_n$ (see, e.g., the survey \cite{MR1757964}), the most powerful one is arguably the \emph{strong Rayleigh property} (\textbf{SRP}) introduced by Borcea, Branden and Liggett \cite{MR2476782}: a probability measure $\pi$ on $\{0,1\}^n$ has the \textbf{SRP} if its generating polynomial
\begin{eqnarray}
\label{def:poly}
(z_1,\ldots,z_n) & \longmapsto & \EE_\pi\left[\prod_{i=1}^nz_i^{X_i}\right]
\end{eqnarray}
has no root $(z_1,\ldots,z_n)\in{\mathbb C}^n$ such that $\mathrm{Im}(z_i)>0$ for all $1\le i\le n$. This fundamental property is robust under all the natural ``closure'' operations -- products, projections, conditioning, external fields, symmetrization, truncations -- and is satisfied by many distributions arising in a variety of contexts. Important examples include determinantal measures, product measures conditioned on their sum, conditional balls-and-bins samples \cite{MR3000852}, and  measures obtained by running the exclusion dynamics from a deterministic state \cite{MR2476782}. 

In 2011, Pemantle and Peres \cite{MR3197973} put forward a strictly weaker form of negative dependence called the \emph{stochastic covering property} (\textbf{SCP}), which will suffice for our purposes. Its definition requires some notation. Let $(\e_i)_{1\le i\le n}$ denote the canonical $n-$dimensional basis. For $x,y\in\{0,1\}^n$, we write $x\triangleright y$  if $x$ can be obtained from $y$ by increasing at most one coordinate from $0$ to $1$, i.e.
\begin{eqnarray}
\label{def:cover}
x=y  & \textrm{or} & \exists i\in[n], x=y+\e_i.
%\end{array}
%\right.
\end{eqnarray}
We then ``lift'' this \emph{covering} relation  to probability measures  on $\{0,1\}^n$ in the usual way, writing  $\pi\triangleright \pi'$ whenever there is a coupling of $\pi$ and $\pi'$ which is supported on the set $\left\{(x,y)\colon x\triangleright y\right\}$.  For $S\subseteq [n]$ and  $x\in\{0,1\}^S$, we abusively write $\pi\left(X_{S^c}=\cdot|X_{S}=x\right)$ for  the conditional law of $(X_i\colon i\in [n]\setminus S)$ given $\left\{\forall i\in S, X_i=x_i\right\}$. Finally, we say that $\pi$ satisfies the \textbf{SCP} if the implication
\begin{eqnarray}
x\triangleright y & \Longrightarrow & \pi\left(X_{S^c}=\cdot|X_{S}=y\right) \triangleright \pi\left(X_{S^c}=\cdot|X_{S}=x\right)
\end{eqnarray}
holds for all choices of $S\subseteq [n]$ and $x,y\in\{0,1\}^S$ such that $\pi\left(X_{S}=x\right),\pi\left(X_{S}=y\right)>0$. This property is strictly weaker than the \textbf{SRP}. More precisely, the following facts are known.
\begin{enumerate}
\item Any measure with the \textbf{SRP} has the \textbf{SCP}  \cite[Proposition 2.2]{MR3197973}.
\item The uniform measure over the bases of a \emph{balanced} matroid has the \textbf{SCP}  \cite[Corollary 3.3]{Feder}. 
\item There are examples of balanced matroids for which the \textbf{SRP} fails \cite{MR2248326}.
\end{enumerate}
The definition of a balanced matroid is not  required for understanding the statements of our main results, and we simply refer  the unfamiliar reader to \cite{MR2065730,Feder, Dyck} for more details. 

\subsection{Results and implications}

\label{sec:main}
Throughout this section, we consider a probability distribution $\pi$ on  $\{0,1\}^n$, and we let
\begin{eqnarray}
\Omega & := & \left\{x\in\{0,1\}^n\colon \pi(x)>0\right\},
\end{eqnarray}
denote its support. We say that $\pi$ is $k-$homogeneous ($0\le k\le n$) if 
\begin{eqnarray}
\label{e:21}
\Omega & \subseteq & \left\{x\in\{0,1\}^n\colon x_1+\cdots+x_n=k\right\}.
\end{eqnarray}
Our starting point is the following celebrated sub-Gaussian concentration inequality for homogeneous measures with the \textbf{SCP}, due to Pemantle \& Peres \cite{MR3197973}. 
\begin{theorem}[Pemantle \& Peres \cite{MR3197973}]\label{th:PP}Let $\pi$ be any $k-$homogeneous distribution on $\{0,1\}^n$ with the \textbf{SCP}, and let $f\colon\{0,1\}^n\to\R$ be any $1-$Lipschitz function. Then, for any $a\ge 0$, we have 
\begin{eqnarray}
\pi\left(f \ge \EE_\pi[f]+a \right) & \le & \exp\left(-\frac{a^2}{8k}\right).
\end{eqnarray}
\end{theorem}
In light of (\ref{conc:MLSI}), it is natural to hope that a much deeper property underlies this remarkable concentration-of-measure phenomenon: namely, that $k-$homogeneous measures with the $\textbf{SCP}$ satisfy a universal modified log-Sobolev inequality with constant $\frac{1}{k}$  w.r.t. some intrinsic local dynamics on the boolean lattice. Our first main contribution consists in showing that this is indeed the case. To formalize the notion of a ``local dynamics'', let us define the relation $\sim$ on $\{0,1\}^n$ as follows:
\begin{eqnarray}
x\sim y & \Longleftrightarrow & x,y \textrm{ differ by a flip or a swap},
\end{eqnarray}
where  $x,y$  differ by a \emph{flip} if $x  = y\pm \e_j$ for some $1\le j\le n$, and by a \emph{swap} if $y  = x+\e_j-\e_k$ for some distinct $1\le j\ne k\le n$ (flips are only relevant for non-homogeneous measures). A \emph{flip-swap walk} for $\pi$ is a  Markov generator $Q$ on $\Omega$ which is reversible under $\pi$ and such that 
\begin{eqnarray}
Q(x,y)>0 & \Longrightarrow & x\sim y\qquad \qquad (x,y\in\Omega).
\end{eqnarray}
We let $\Delta(Q):=\max\left\{-Q(x,x)\colon x\in\Omega\right\}$ denote the maximal rate at which changes occur under $Q$, and we say that $Q$ is \emph{normalized} whenever $\Delta(Q)\le 1$. This means that $Q$ may be written as
\begin{eqnarray}
\label{discrete}
Q & = & I-P,
\end{eqnarray}
for some stochastic matrix $P$, allowing us to encompass the (more standard) discrete-time setting. With this terminology in hands, our main result asserts that every $k-$homogeneous measure with the \textbf{SCP} admits a normalized flip-swap walk with modified log-Sobolev constant $\frac{1}{2k}$. 
\begin{theorem}[Intrinsic modified log-Sobolev inequality for \textbf{SCP} measures]\label{th:conc}If $\pi$ is a $k-$homogeneous measure with the \textbf{SCP}, then there is a normalized flip-swap walk $Q$ for $\pi$  such that
\begin{eqnarray}
\lambda(Q),\alpha(Q) & \ge &\frac 1{2k}.
\end{eqnarray}
Furthermore, the conclusion remains valid without the homogeneity assumption, with $k:=\frac n 2$.
\end{theorem}

By virtue of (\ref{conc:MLSI}), this immediately implies Theorem \ref{th:PP}, albeit with the constant $8$ replaced by $32$. As already explained, modified log-Sobolev inequalities are of intrinsic interest beyond the sole fact that they imply sub-Gaussian concentration. For example, another notable consequence of Theorem \ref{th:conc}  concerns  sampling complexity: starting from an arbitrary initial state $x\in\Omega$, our flip-swap walk will -- by virtue of (\ref{mixing:MLSI}) --  produce an $\varepsilon-$approximate sample from $\pi$ in time
\begin{eqnarray}
\tmix(x;\varepsilon) & \le & 2k\left(\log\log\frac{1}{\pi(x)}+\log\frac{1}{2\varepsilon^2}\right).
\end{eqnarray}
Here again, the conclusion remains valid without the homogeneity assumption, provided we set $k:=\frac n 2$. Note that this bound -- and hence our modified log-Sobolev estimate -- is sharp up to a factor of $4$ only, as can  be seen by considering the basic case where $\pi$ is uniform on $\{0,1\}^n$. 

In practical situations however, one might be constrained to use more tractable rates than those of the flip-swap walk  of Theorem \ref{th:conc}. A canonical choice is the Monte Carlo Markov Chain (\textbf{MCMC})
\begin{eqnarray}
\label{MCMC}
Q(x,y) & := & 
%\left\{
\frac{1}{2kn}\min\left\{\frac{\pi(y)}{\pi(x)},1\right\} \quad \textrm{ if }x\sim y.
%- \frac{1}{2}\\
%0 & \textrm{else}.  
%\right.
\end{eqnarray}
Notice that only the positive entries have here been made explicit: all other off-diagonal entries should be interpreted as zero, and the diagonal adjusted so that the rows sum up to zero. The normalization  ensures that $\Delta(Q)\le \frac 12$, so that $Q$ corresponds to a lazy stochastic matrix $P$ via (\ref{discrete}). The strategy that we develop to establish Theorem \ref{th:conc} is sufficiently robust to yield sharp performance guarantees for this chain too, and many others. For any generator $Q$ on $\Omega$, define 
\begin{eqnarray}
\bm(Q) & := &\min_{\tiny\begin{array}{cc}x,y\in\Omega\\x\sim y\end{array}} Q(x,y),\\
\bM(Q) & := & \min_{\tiny\begin{array}{cc}x,y\in\Omega\\x\sim y\end{array}}  \max\left\{Q(x,y),Q(y,x)\right\}.
\end{eqnarray}
Note that $\bm(Q)\le \bM(Q)$. It turns out that these simple statistics provide universal functional-analytic estimates whenever $Q$ is reversible w.r.t. a measure with the $\textbf{SCP}$. 
\begin{theorem}[Universal functional inequalities for \textbf{SCP} measures]\label{th:main}Consider a probability distribution $\pi$ with the \textbf{SCP}, and let $Q$ be {any} reversible Markov generator with respect to $\pi$. Then,
\begin{eqnarray}
\label{th:PI}\lambda(Q) & \ge & \max\left\{\bM(Q),2\bm(Q)\right\};\\
\label{th:MLSI}\alpha(Q) & \ge & \max\left\{\bM(Q),4\bm(Q)\right\};\\
\label{th:LSI}\rho(Q) & \ge & \bm(Q).
\end{eqnarray}
\end{theorem}
This result unifies, extends or strengthens various powerful estimates obtained over the past decades. In particular, in an inspiring work, Jerrum \& Son \cite{1181997} considered the special case where $\Omega$ is the set of bases of a balanced matroid of rank $k$, and $\pi$ is the uniform distribution on it. In this setting, the \textbf{MCMC} (\ref{MCMC}) reduces to the so-called \emph{bases-exchange walk}
\begin{eqnarray}
\label{def:mat}
Q(x,y) & := & \frac{1}{2kn}\quad \textrm{ if }x\sim y.
\end{eqnarray}
For the latter, Jerrum \& Son established the Poincaré and log-Sobolev inequalities
\begin{eqnarray}
\label{JS}
\lambda(Q)\ \ge\ \frac{1}{kn}, & \qquad & \rho(Q)\ \ge\ \frac{1}{4kn}.
\end{eqnarray}
Both estimates  readily follow from  Theorem \ref{th:main}, the second one being even improved by a factor $2$. 

Note that the above setting is rather specific, since $\pi$ is uniform. 
In a recent breakthrough, Anari, Oveis Gharan \& Rezaei \cite{anari2016monte} managed to relax this constraint: under the assumption that $\pi$ is $k-$homogeneous with the $\textbf{SRP}$, they showed that the chain (\ref{MCMC}) satisfies the Poincaré inequality
\begin{eqnarray}
\label{PI:MCMC}
\lambda(Q) & \ge & \frac{1}{2kn}.
\end{eqnarray}
By virtue of (\ref{mixing:PI}), this immediately implies the mixing-time bound
\begin{eqnarray}
\label{anari}
\tmix(x;\varepsilon) & \le & kn\left(\log\frac{1}{\pi(x)}+\log\frac{1}{4\varepsilon^2}\right).
\end{eqnarray}
This result is again covered by  Theorem \ref{th:main},  under the  weaker condition that $\pi$ has the \textbf{SCP}. More importantly, Theorem \ref{th:main} provides a new modified log-Sobolev inequality with the same constant, leading to a much tighter control on the convergence rate. Let us state this as a corollary.
\begin{corollary}[Modified log-Sobolev inequalities for \textbf{MCMC}]\label{coro:MLSI}
Consider a $k-$homogeneous measure $\pi$ on $\{0,1\}^n$ with the $\textbf{SCP}$, and let $Q$ be the  $\textbf{MCMC}$ defined at (\ref{MCMC}). Then,
\begin{eqnarray}
\label{co:MLSI}\alpha(Q) & \ge & \frac{1}{2kn}.
\end{eqnarray}
In particular, for any initial state $x\in\Omega$ and any precision $\varepsilon\in(0,1)$, we have
\begin{eqnarray}
\tmix(x;\varepsilon) & \le & 2kn\left(\log\log\frac{1}{\pi(x)}+\log\frac{2}{\varepsilon^2}\right),
\end{eqnarray}
which offers a considerable improvement upon (\ref{anari}).  
\end{corollary}

We note that resorting to the \emph{modified} log-Sobolev constant is crucial here: when the measure $\pi$ is not uniform,  the log-Sobolev constant $\rho(Q)$  can  not be bounded from below by a function of $n,k$ only, as in (\ref{JS}). Indeed, it follows from the definitions that any normalized generator $Q$  satisfies
\begin{eqnarray}
\rho(Q) & \le & \min_{x\in\Omega}\left\{\left(\log\frac 1{\pi(x)}\right)^{-1}\right\},
\end{eqnarray}
and the right-hand side can be arbitrarily small without further assumptions on $\pi$.

\begin{remark}[Strong log-concavity]Very recently, a new powerful notion of negative dependence for $k-$homogenous measures on $\{0,1\}^n$ called \emph{strong log-concavity}  (\textbf{SLC}) has been  developed in two impressive series of papers by Anari, Liu, Oveis-Gharan  \&  Vinzant \cite{LCP1, LCP2, LCP3} and by Bränd\'en and Huh \cite{brndn2018hodgeriemann, BJ}. In particular, three months before the appearance of our paper online, a Poincaré inequality with constant $1/k$ was established by Anari, Liu, Oveis-Gharan  \&  Vinzant \cite{LCP2} for sampling from such measures, and a month after the appearance of our paper online, this Poincaré inequality was improved to a \textbf{MLSI} with constant $1/k$ by Cryan, Guo \& Mousa \cite{logcon}. We note, however, that this powerful new result is not directly  comparable to ours, for several reasons. First,  the negative dependence assumptions $\textbf{SLC}$ and $\textbf{SCP}$ are not comparable, in the sense that neither of them implies the other, as observed in \cite{logcon}. Second, the \textbf{MCMC} considered in \cite{LCP2, logcon} is actually an accelerated version of  (\ref{MCMC}), in which each iteration may require more than constant time to implement in practice.  
Third, we recall that neither homogeneity, nor the specific form (\ref{MCMC}) is required by our methods: Theorem \ref{th:main}  applies to \emph{any} measure $\pi$ with the $\textbf{SCP}$, and \emph{any} reversible  generator $Q$ w.r.t. $\pi$.               
\end{remark}           

\section{Proofs}
\subsection{A general decomposition lemma}
\label{sec:decomposition}
Our starting point is a well-known recursive strategy for establishing functional inequalities for reversible chains. The method was invented by Lu \& Yau in the context of interacting particle systems \cite{MR1233852}, and developed further by Jerrum \& Son \cite{1181997} and Jerrum, Son, Tetali \& Vigoda \cite{MR2099650}. 

Let $Q$ be any reversible Markov generator with respect to some (fully-supported) probability distribution $\pi$ on a finite set $\Omega$. Consider an arbitrary partition of the state space:
\begin{eqnarray}
\label{partition}
\Omega & = & \bigcup_{i\in \cI} \Omega_i.
\end{eqnarray}
The \emph{projection chain} induced by this partition is the Markov chain whose state space is $\cI$ and whose Markov generator $\oQ$ is defined as follows: for all $(i,j)\in\Omega^2$,
\begin{eqnarray}
\oQ(i,j) & := & \frac{1}{\opi(i)}\sum_{x\in\Omega_i}\sum_{y\in\Omega_j}\pi(x)Q(x,y),\qquad \textrm{where}\qquad \opi(i) \ := \ \sum_{x\in\Omega_i}\pi(x).
\end{eqnarray}
Note that $\opi$ is a probability distribution on $\cI$, and that it is reversible under $\oQ$, i.e.
\begin{eqnarray}
\opi(i)\oQ(i,j) & = &\opi(j)\oQ(j,i).
\end{eqnarray}
For each $i\in\cI$, the \emph{restriction chain} on $\Omega_i$ is the Markov chain whose state space is $\Omega_i$ and whose Markov generator $Q_i$ is defined as follows: for any distinct states $x,y$ in $\Omega_i$,
\begin{eqnarray}
Q_i(x,y) & = & 
Q(x,y),
\end{eqnarray}
with the diagonal entries being adjusted so that the rows of $Q_i$ sum up to zero.
It is then clear that the probability measure $\pi_i$ defined on $\Omega_i$ by
\begin{eqnarray}
\label{def:restriction}
\pi_i(x) &:= & \frac{\pi(x)}{\opi(i)}
\end{eqnarray}
is reversible under $Q_i$. Now, suppose that for each $(i,j)\in \cI^2$ with $\oQ(i,j)>0$, we are given a coupling $\kappa_{ij}\colon\Omega_i\times\Omega_j\to [0,1]$ of the probability distributions $\pi_i$ and $\pi_j$, i.e.
\begin{eqnarray}
\label{coupling:x}
\forall x\in\Omega_i,\quad \sum_{y\in\Omega_j}\kappa_{ij}(x,y) & = & \pi_i(x),\\
\label{coupling:y}
\forall y\in\Omega_j,\quad \sum_{x\in\Omega_i}\kappa_{ij}(x,y) & = & \pi_j(y).
\end{eqnarray}
For reasons that will soon become clear, we measure the \emph{quality} of these couplings by the quantity
\begin{eqnarray}
\label{def:chi}
\chi & := & \min\left\{\frac{\pi(x)Q(x,y)}{\opi(i)\oQ(i,j)\kappa_{ij}(x,y)}\right\},
\end{eqnarray}
where the minimum runs over all $(x,y,i,j)$ such that the denumerator is positive. Note that $\chi=0$ if $\min_{x \in i } \sum_{y \in j}Q(x,y)=0$ for some $i,j \in \mathcal{I}$ with $\oQ(i,j)>0$. Hence such partitions should be avoided when applying Lemma \ref{lm:main} below. 
\begin{lemma}[Recursive functional inequalities]\label{lm:main}With the above notations, we have
\begin{eqnarray}
\label{rec:lambda}\lambda(Q) & \ge & \min\left\{\chi\lambda(\oQ),\min_{i\in\cI}\lambda(Q_i)\right\};\\
\label{rec:alpha}\alpha(Q) & \ge & \min\left\{\chi\alpha(\oQ),\min_{i\in\cI}\alpha(Q_i)\right\};\\
\label{rec:rho}\rho(Q) & \ge & \min\left\{\chi\rho(\oQ),\min_{i\in\cI}\rho(Q_i)\right\}.
\end{eqnarray}
\end{lemma}

\begin{proof} Let $f\colon\Omega\to(0,\infty)$. The left-hand sides of (\ref{PI}),(\ref{MLSI}),(\ref{LSI}) take the generic form 
\begin{eqnarray}
\label{def:L}
\cL_\pi(f) & := & \frac{1}{2}\sum_{(x,y)\in \Omega^2}\pi(x)Q(x,y)\Psi\left(f(x),f(y)\right),
\end{eqnarray}
with the function $\Psi\colon (0,\infty)^2\to[0,\infty)$ being given by
\begin{eqnarray}
\label{def:Psi}
\Psi(u,v) & := & \left\{
\begin{array}{ll}
(u-v)^2 & \textrm{in the Poincaré case}\\
(u-v)\left(\log u-\log v\right) & \textrm{in the modified log-Sobolev case}\\
(\sqrt{u}-\sqrt{v})^2 & \textrm{in the log-Sobolev case}.
\end{array}
\right.
\end{eqnarray}
The quantity $\cL_\pi(f)$ measures the average \emph{local} variation of $f$ along the transitions of the chain. We wish to compare it to the average \emph{global} variation of $f$ across the whole space, as measured by 
\begin{eqnarray}
\cR_\pi(f) & := & \left\{
\begin{array}{ll}
\Var_{\pi}(f) & \textrm{in the Poincaré case}\\
\ent_\pi(f) & \textrm{in the log-Sobolev and modified log-Sobolev cases}.
\end{array}
\right.
\end{eqnarray}
To this end, we start by decomposing $\cR_\pi(f),\cL_\pi(f)$ according to the partition (\ref{partition}). Define a projected observable $\of\colon \cI\to(0,\infty)$ by $\of(i) := \EE_{\pi_i}(f)$. It is immediate to check that
\begin{eqnarray}
\label{DEC1}
\cR_\pi(f) & = & \sum_{i\in\cI}\opi(i)\cR_{\pi_i}(f)+\cR_{\opi}(\of);\\
\label{DEC2}
\cL_\pi(f) & = & \sum_{i\in\cI}\opi(i)\cL_{\pi_i}(f)+\frac{1}{2}\sum_{i\ne j}\sum_{(x,y)\in\Omega_i\times\Omega_j}\pi(x)Q(x,y)\Psi\left(f(x),f(y)\right).
\end{eqnarray} 
In light of the similarity between the right-hand sides, it is then natural to hope for a term-by-term comparison. The problematic quantity is of course the second sum on the right-hand side of (\ref{DEC2}), which we would like to replace by  $\cL_{\opi}(\of)$. To this end, let us fix a pair $(i,j)\in\cI^2$ such that $\oQ(i,j)>0$. Since $\kappa_{ij}$ is a coupling of $\pi_i$ and $\pi_j$, we have
\begin{eqnarray}
\of(i) & = & \sum_{(x,y)\in \Omega_i\times\Omega_j}\kappa_{ij}(x,y)f(x),\\
\of(j) & = & \sum_{(x,y)\in \Omega_i\times\Omega_j}\kappa_{ij}(x,y)f(y).
\end{eqnarray}
Thus, the pair $(\of(i),\of(j))$ is the average of the function $(x,y)\mapsto(f(x),f(y))$ under the probability distribution $\kappa_{ij}$. The crucial observation is that the bivariate function $\Psi$ defined at  (\ref{def:Psi}) is convex, as can be easily checked by verifying that the hessian matrix
\begin{eqnarray}
H & := &
\left[
\begin{array}{cc}
\partial_{uu}\Psi & \partial_{uv}\Psi  \\
\partial_{vu}\Psi  & \partial_{vv}\Psi
\end{array}
\right],
\end{eqnarray}
is positive semi-definite  in each case. Therefore, Jensen's inequality ensures that
\begin{eqnarray}
\sum_{(x,y)\in \Omega_i\times\Omega_j}\kappa_{ij}(x,y)\Psi(f(x),f(y)) & \ge & \Psi\left(\of(i),\of(j)\right).
\end{eqnarray}
Multiplying through by $\chi \opi(i)\oQ(i,j)$ and recalling the definition of $\chi$, we deduce that
\begin{eqnarray}
\sum_{(x,y)\in\Omega_i\times\Omega_j}\pi(x)Q(x,y)\Psi\left(f(x),f(y)\right) & \ge & {\chi}\opi(i)\oQ(i,j)\Psi\left(\of(i),\of(j)\right).
\end{eqnarray}
Although this was established under the assumption that $\oQ(i,j)>0$, the conclusion  remains trivially valid when $\oQ(i,j)=0$. Summing over all pairs $(i,j)\in\cI^2$ with $i\ne j$ and plugging the resulting estimate back into (\ref{DEC2}), we arrive at
\begin{eqnarray}
\cL_{\pi}\left(f\right) & \ge & \sum_{i\in\cI}\opi(i)\cL_{\pi_i}(f)+\chi \cL_{\opi}(\of).
\end{eqnarray}
Comparing this with (\ref{DEC1}), we immediately deduce that for any $\ok,\kappa_i\ge 0$,
\begin{eqnarray}
\left\{
\begin{array}{l}
\cL_{\opi}(\of) \ge \ok\cR_{\opi}(\of)\\
\forall i\in\cI, \cL_{\pi_i}(f) \ge \kappa_i\cR_{\pi_i}(f)
\end{array}
\right.
& \Longrightarrow & 
\cL_\pi(f) \ge  \min\left\{\chi\ok,\min_{i\in\cI} \kappa_i\right \}\cR_\pi(f).
\end{eqnarray}
Since this implication holds for all observables $f\colon\Omega\to(0,\infty)$, the three claims follow.
\end{proof}
We end this section with a crude estimate on the ratio appearing in the definition of $\chi$.
\begin{lemma}[Crude lower-bound on $\chi$]\label{lm:crude}We always have 
\begin{eqnarray}
\frac{\pi(x)Q(x,y)}{\opi(i)\oQ(i,j)\kappa_{ij}(x,y)} & \ge & \max\left\{\frac{Q(x,y)}{\oQ(i,j)},\frac{Q(y,x)}{\oQ(j,i)}\right\},
%\\ & \ge & \max\left\{
%\frac{\max\left(Q(x,y),Q(y,x)\right)}{\max\left(\oQ(i,j),\oQ(j,i)\right)},
%\frac{\min\left(Q(x,y),Q(y,x)\right)}{\min\left(\oQ(i,j),\oQ(j,i)\right)}
%\right\}.
\end{eqnarray}
provided the denumerator on the left-hand side is positive.
\end{lemma}
\begin{proof}
Since $\kappa_{ij}$ is a coupling of $\pi_i$ and $\pi_j$, we have $\kappa_{ij}(x,y)\le \pi_i(x)$ and $\kappa_{ij}(x,y)\le \pi_j(y)$. Thus,
\begin{eqnarray}
\frac{\pi(x)Q(x,y)}{\opi(i)\oQ(i,j)\kappa_{ij}(x,y)}  & \ge & \frac{Q(x,y)}{\oQ(i,j)},\\
\frac{\pi(y)Q(y,x)}{\opi(j)\oQ(j,i)\kappa_{ij}(x,y)} & \ge & \frac{Q(y,x)}{\oQ(j,i)}.
\end{eqnarray}
The claim follows by noting that the left-hand sides of these two lines are equal, by reversibility.
\end{proof}

\subsection{Proof of Theorem \ref{th:main}}

We are now in position to prove Theorem \ref{th:main} by induction over the number of states $|\Omega|$. The claim is trivial when $|\Omega|=1$. Now, consider a probability distribution $\pi$ on $\{0,1\}^n$  whose support  $\Omega$ has at least two elements. This means that there is an index $\ell\in\{1, \ldots, n\}$ such that the subsets
\begin{eqnarray}
\Omega_0 := \left\{x\in\Omega\colon x_\ell=0\right\}, & \qquad & 
\Omega_1 := \left\{x\in\Omega\colon x_\ell=1\right\},
\end{eqnarray}
are both non-empty. Let $Q$ be any Markov generator on $\Omega$ under which $\pi$ is reversible, and let $(Q_0,\Omega_0,\pi_0)$, $(Q_1,\Omega_1,\pi_1)$ and  $(\overline{Q},\{0,1\},\opi)$ be the restriction and projection chains induced by the partition $\Omega=\Omega_0\cup\Omega_1$. Using the short-hands $a:=\oQ(0,1)$ and $b:=\oQ(1,0)$, we have
\begin{eqnarray}
\label{def:twostate}
\oQ & = & \left[
\begin{array}{cc}
-a & a  \\
b  & -b
\end{array}
\right].
\end{eqnarray}
Functional-analytic estimates for such two-state chains can easily be found in the literature. 
\begin{lemma}[Two-state chains, see, e.g. \cite{MR2283379}]\label{lm:twostate}For any choice of the rates $a,b\ge 0$, we have
\begin{eqnarray}
\lambda(\oQ) \ = \ a+b, \qquad
\alpha(\oQ) \ \in \ \left[a+b,2(a+b)\right], \qquad
\rho(\oQ) \ = \ 
\left\{
\begin{array}{ll}
\frac{a-b}{\log a-\log b} & \textrm{if }a\ne b\\
a & \textrm{if } a=b.
\end{array}
\right.
\end{eqnarray}
Since $\left|a-b\right| \ge \min(a,b)\left|\log a-\log b\right|$, the last estimate implies in particular 
$
\rho(\oQ)  \ge \min(a,b).
$
\end{lemma}
 In order to apply Lemma \ref{lm:main}, it now remains to construct an appropriate coupling of $\pi_0$ and $\pi_1$. This is provided by the following lemma, which is where the \textbf{SCP}  comes into play.  
\begin{lemma}[Exploiting the \textbf{SCP}]\label{lm:SCP}If $\pi$ has the \textbf{SCP}, there is a coupling $\kappa$ of $\pi_0,\pi_1$  supported on 
\begin{eqnarray}
\left\{(x,y)\in\Omega_0\times\Omega_1\colon x\sim y\right\}.
\end{eqnarray}
In particular, taking $\kappa_{01}$ and $\kappa_{10}$ to be $\kappa$ and its transpose, respectively, we obtain
\begin{eqnarray}
\chi & \ge & \max\left\{\frac{\bM(Q)}{\max\left(a,b\right)},\frac{\bm(Q)}{\min\left(a,b\right)}\right\}.
\end{eqnarray}
\end{lemma}
\begin{proof}For ease of notation, let us here assume that $\ell=n$. The  \textbf{SCP} ensures that we can construct, on a common probability space, two random vectors $(X_1,\ldots,X_{n-1})$ and $(Y_1,\ldots,Y_{n-1})$  such that
\begin{enumerate}
\item $\left(X_1,\ldots,X_{n-1},0\right)$ has law $\pi_0$ ;
\item $\left(Y_1,\ldots,Y_{n-1},1\right)$ has law $\pi_1$ ;
\item $\left(X_1,\ldots,X_{n-1}\right)\triangleright \left(Y_1,\ldots,Y_{n-1}\right)$ with probability $1$.
\end{enumerate}
By definition of $\triangleright$ and $\sim$, the property $\left(X_1,\ldots,X_{n-1}\right)\triangleright \left(Y_1,\ldots,Y_{n-1}\right)$ deterministically implies
\begin{eqnarray}
\left(X_1,\ldots,X_{n-1},0\right) & \sim & \left(Y_1,\ldots,Y_{n-1},1\right).
\end{eqnarray}
The joint law $\kappa$ of these two vectors is thus a coupling of $\pi_0$ and $\pi_1$ satisfying the claim. To estimate the resulting constant $\chi$, we simply invoke Lemma \ref{lm:crude} to get
\begin{eqnarray}
\chi & \ge & \min\left\{\max\left\{\frac{Q(x,y)}{a},\frac{Q(y,x)}{b}\right\}\colon{(x,y)\in\Omega_0\times\Omega_1}, x\sim y\right\}.
\end{eqnarray}
The second claim now readily follows from this and the definitions of $\bm(Q),\bM(Q)$. 
\end{proof}
We now have all we need to complete our induction step. Using Lemmas \ref{lm:twostate} and \ref{lm:SCP}, we find
\begin{eqnarray}
\chi\lambda(\oQ)  & \ge  & \bM(Q),\\
\chi\alpha(\oQ)& \ge &  \bM(Q),\\
\chi\rho(\oQ) & \ge  & \bm(Q).
\end{eqnarray}
Applying Lemma \ref{lm:main}, we deduce that
\begin{eqnarray}
\lambda(Q) & \ge & \min\left\{\bM(Q),\lambda(Q_0),\lambda(Q_1)\right\};\\
\alpha(Q) & \ge & \min\left\{\bM(Q),\alpha(Q_0),\alpha(Q_1)\right\};\\
\rho(Q) & \ge & \min\left\{\bm(Q),\rho(Q_0),\rho(Q_1)\right\}.
\end{eqnarray}
Now, observe that the \textbf{SCP} is closed under conditioning, so that $\pi_0$ and $\pi_1$ inherit it from $\pi$. Thus, the induction hypothesis applies to the restriction chains $(\Omega_i,\pi_i,Q_i)$, $i\in\{0,1\}$, yielding
\begin{eqnarray}
\lambda(Q_i)  & \ge &  \bM(Q_i),\\
\alpha(Q_i)  & \ge &  \bM(Q_i),\\
\rho(Q_i)  & \ge &  \bm(Q_i).
\end{eqnarray}
Since we obviously have $\bM(Q_i)\ge \bM(Q)$ and $\bm(Q_i)\ge\bm(Q)$, we  conclude that
\begin{eqnarray}
\lambda(Q) &  \ge  & \bM(Q),\\
\alpha(Q) &  \ge &  \bM(Q),\\
\rho(Q) &  \ge &  \bm(Q).
\end{eqnarray}
Finally, the slight improvement appearing in (\ref{th:PI})-(\ref{th:MLSI}) simply follows from  (\ref{hier}).

\subsection{Proof of Theorem \ref{th:conc}} We now prove Theorem \ref{th:conc} by induction over the dimension $n$. It will actually be slightly more convenient to change the normalization as follows. For any distribution $\pi$ on $\{0,1\}^n$ with the \textbf{SCP}, we will prove the existence of a flip-swap walk $Q$ for $\pi$ such that $\lambda(Q),\alpha(Q)\ge 1$ and
\begin{eqnarray}
\Delta(Q) & \le & \left\{
\begin{array}{ll}
n & \textrm{always}\\
2k & \textrm{if }\pi\textrm{ is }k-\textrm{homogeneous}.
\end{array}
\right.
\end{eqnarray}
The original claim is obtained by dividing all entries by $\Delta(Q)$. The base case $n=1$ is trivial, and we henceforth assume that $n\ge 2$. Fix a coordinate $\ell\in\{1,\ldots,n\}$. We may assume that the sets
\begin{eqnarray}
\label{def:l}
\Omega_0 := \left\{x\in\Omega\colon x_\ell=0\right\}, & \qquad & 
\Omega_1 := \left\{x\in\Omega\colon x_\ell=1\right\},
\end{eqnarray}
are both non-empty: otherwise, $\pi$ can be regarded as a measure on $\{0,1\}^{n-1}$, and the claim trivially follows from the induction hypothesis. Now, consider the corresponding projection and restriction measures $\opi, \pi_0,\pi_1$, as defined in Section \ref{sec:decomposition}. For $i\in\{0,1\}$, $\pi_i$ may be regarded as a measure on $\{0,1\}^{n-1}$, and it inherits the \textbf{SCP} from $\pi$. Notice that if $\pi$ is $k-$homogeneous, then so is $\pi_0$, while $\pi_1$ is $(k-1)-$homogeneous (when regarded as a measure on $\{0,1\}^{n-1}$).  Thus, for $i\in\{0,1\}$, the induction hypothesis provides a flip-swap walk $Q_i$ for $\pi_i$ satisfying $\lambda(Q_i),\alpha(Q_i)\ge 1$ and
\begin{eqnarray}
\label{induction}
\Delta(Q_i) & \le & \left\{
\begin{array}{ll}
n-1 & \textrm{always}\\
2(k-i) & \textrm{if }\pi\textrm{ is }k-\textrm{homogeneous}.
\end{array}
\right.
\end{eqnarray}
On the other hand, Lemma \ref{lm:SCP} provides a coupling $\kappa$ of $\pi_0$ and $\pi_1$  supported on the set 
\begin{eqnarray}
\left\{(x,y)\in\Omega_0\times\Omega_1\colon x\sim y\right\}.
\end{eqnarray}
With these ingredients in hands, we  define a flip-swap walk $Q$ for $\pi$ as follows: for $x\ne y\in\Omega$,
\begin{eqnarray}
\label{def:Q}
Q(x,y) & := & 
\left\{
\begin{array}{ll}
Q_0(x,y) & \textrm{if }(x,y)\in\Omega_0\times\Omega_0\medskip\\
Q_1(x,y) & \textrm{if }(x,y)\in\Omega_1\times\Omega_1\medskip\\
\displaystyle{\frac{\opi(0)\opi(1)\kappa(x,y)}{\pi(x)}} & \textrm{if }(x,y)\in\Omega_0\times\Omega_1\medskip\\
\displaystyle{\frac{\opi(0)\opi(1)\kappa(y,x)}{\pi(y)}} & \textrm{if }(x,y)\in\Omega_1\times\Omega_0.
\end{array}
\right.
\end{eqnarray}
Here again, the diagonal is implicitly adjusted so that the rows sum up to $0$.
In order to  estimate $\lambda(Q),\alpha(Q)$, we will now use Lemma \ref{lm:main} with the partition being $\Omega=\Omega_0\cup\Omega_1$ and the couplings $\kappa_{01}$ and $\kappa_{10}$ being $\kappa$ and its transpose, respectively. By Lemma \ref{lm:twostate}, the projection chain $\oQ$ satisfies
\begin{eqnarray}
\label{two}
\lambda(\oQ),\alpha(\oQ) & \ge & \oQ(0,1)+\oQ(1,0).
\end{eqnarray}
On the other hand, for any $(x,y)\in\Omega_0\times\Omega_1$ with $\kappa(x,y)>0$, we have by construction,
\begin{eqnarray}
\frac{\pi(x)Q(x,y)}{\opi(0)\oQ(0,1)\kappa_{01}(x,y)} & = & \frac{\opi(1)}{\oQ(0,1)} \ = \ \frac{\opi(0)}{\oQ(1,0)} \ =  \ \frac{\pi(y)Q(y,x)}{\opi(1)\oQ(1,0)\kappa_{10}(y,x)},
\end{eqnarray}
where the equality in the middle is nothing more than reversibility. Consequently, we see that 
\begin{eqnarray}
\chi & = & \frac{\opi(0)}{\oQ(1,0)}\ =\ \frac{\opi(1)}{\oQ(0,1)}.
\end{eqnarray}
Combining this with (\ref{two}), we deduce that
\begin{eqnarray}
\chi\lambda(\oQ),\chi\alpha(\oQ) & \ge & \opi(0)+\opi(1) \ = \ 1.
\end{eqnarray}
Recalling that $\lambda(Q_i),\alpha(Q_i)\ge 1$, we may finally invoke Lemma \ref{lm:main} to conclude that 
\begin{eqnarray}
\lambda(Q),\alpha(Q) & \ge & 1.
\end{eqnarray}
It now remains to estimate the diagonal entries. For $x\in\Omega_0$, we have by construction
\begin{eqnarray}
-Q(x,x)  & = & \sum_{y\in\Omega_0\setminus\{x\}}Q_0(x,y)+\frac{\opi(0)\opi(1)}{\pi(x)}\sum_{y\in\Omega_1}\kappa(x,y).
\end{eqnarray}
Using the definition of $\Delta(Q_0)$ and the fact that the first marginal of $\kappa$ is $\pi_0$, we deduce that
\begin{eqnarray}
\label{diag:0}
-Q(x,x)  &  \le & \Delta(Q_0)+\opi(1) \ = \ \Delta(Q_0)+\EE_\pi[X_\ell].
\end{eqnarray}
Similarly, if  $x\in\Omega_1$, we have 
\begin{eqnarray}
\label{diag:1}
-Q(x,x)  &  \le & \Delta(Q_1)+\opi(0) \ \le \ \Delta(Q_1)+x_\ell,
\end{eqnarray}
because $x_\ell=1$ in this case. Using (\ref{induction}), we deduce that for all $x\in\Omega$,
\begin{eqnarray}
\label{diag}
-Q(x,x) & \le & \left\{
\begin{array}{ll}
n & \textrm{always}\\
2k+\EE[X_\ell]-x_\ell & \textrm{if }\pi\textrm{ is }k-\textrm{homogeneous},
\end{array}
\right.
\end{eqnarray}
Unfortunately, the second line is \emph{not} bounded from above by $2k$. To fix this, our last step will consist in averaging over the choice of the coordinate $\ell$ used to define the partition (\ref{def:l}). We henceforth write $Q=Q^{(\ell)}$ to indicate explicitly the dependency upon   $\ell$. For each $\ell\in\{1,\ldots,n\}$, the construction (\ref{def:Q}) produces a flip-swap walk  $Q^{(\ell)}$ for $\pi$ satisfying $\lambda(Q^\ell),\alpha(Q^\ell)\ge 1$ and $\Delta(Q^\ell)\le n$. Since all these properties are preserved under convex combinations,  the generator
\begin{eqnarray}
Q^\star & := & \frac{1}{n}\sum_{\ell=1}^nQ^{(\ell)}
\end{eqnarray}
automatically inherits them. Now, if $\pi$ is $k-$homogeneous, then for each $x\in\Omega$, we have
\begin{eqnarray}
-Q^{\star}(x,x)  \ = \ -\frac{1}{n}\sum_{\ell=1}^nQ^{(\ell)}(x,x) & \le & \frac{1}{n}\sum_{\ell=1}^n\left(2k+\EE_\pi\left[X_\ell\right]-x_\ell\right)\
 =  \ 2k.
\end{eqnarray}
where the inequality follows from (\ref{diag}), and the equality from the definition of $k-$homogeneity. Thus, the generator $Q^\star$ enjoys all the desired properties, and this completes the induction step.
\subsection{The Herbst argument}
\label{sec:herbst}
For completeness, we finally explicitate the \emph{Herbst argument} used to turn any modified log-Sobolev inequality into a sub-Gaussian concentration estimate, as claimed at (\ref{conc:MLSI}). We emphasize that the approach is standard, and has been used in various settings, see e.g. \cite{MR1767995,MR3185193,MR2283379}. We also note that similar concentration results can  be derived via transportation inequalities, as observed by Sammer \cite[Prop.\ 3.1.3 and \S1.4.1]{Sammer} and by Sammer and Tetali \cite{ST}.
\begin{lemma}[Modified log-Sobolev implies Gaussian concentration]
Let $Q$ be a reversible Markov generator with respect to some probability distribution $\pi$ on a finite set $\Omega$. Then, 
\begin{eqnarray}
\pi\left(f \ge \EE_\pi[f]+a\right) & \le & \exp\left(-\frac{\alpha(Q)a^2}{4v(f)}\right),
\end{eqnarray}
for any observable $f\colon\Omega\to\R$ and any $a \ge 0$,  where 
\begin{eqnarray}
\label{def:v}
v(f) & := & \max_{x\in\Omega}\left\{\sum_{y\in\Omega}Q(x,y)\left[f(y)-f(x)\right]^2_+\right\}. 
\end{eqnarray}
\end{lemma}
\begin{proof}To lighten notations, we drop the subscript $\pi$.  For $t\in(0,\infty)$ and $x\in\Omega$, define
\begin{eqnarray}
F_t(x) & := & e^{tf(x)-ct^2},
\end{eqnarray}
where $c>0$ will be adjusted later.  Using the reversibility $\pi(x)Q(x,y)=\pi(y)Q(y,x)$, we have
\begin{eqnarray}
\cE\left(F_t,\log F_t\right) 
& = &  \frac{t}{2} \sum_{(x,y)\in \Omega^2}\pi(x)Q(x,y)\left(F_t(x)-F_t(y)\right)\left(f(x)-f(y)\right)\\
& = & t \sum_{(x,y)\in \Omega^2}\pi(x)Q(x,y)F_t(x)\left(1-e^{-t\left(f(x)-f(y)\right)}\right)\left[f(x)-f(y)\right]_+\\
& \le & t^2\sum_{(x,y)\in \Omega^2}\pi(x)Q(x,y)F_t(x) \left[f(x)-f(y)\right]^2_+\\
& \le & t^2 v(f)\EE[F_t],
\end{eqnarray}
thanks to the bound $1-e^{-u}\le u$ and (\ref{def:v}). 
Recalling the definition of $\alpha(Q)$, we deduce that
\begin{eqnarray}
\label{herbst}
\ent(F_t) & \le & \frac{v(f)}{\alpha(Q)}t^2\EE[F_t].
\end{eqnarray}
On the other hand, for any $t>0$, we easily compute
\begin{eqnarray}
\frac{d}{dt}\left\{\frac{\log\EE\left[F_t\right]}{t}\right\} & = & \frac{\ent\left[F_t\right]-{ct^2}\EE[F_t]}{t^2\EE\left[F_t\right]}.
\end{eqnarray}
Choosing $c:=\frac{v(f)}{\alpha(Q)}$, we see that $t\mapsto \frac {\log\EE\left[F_t\right]}{t}$ is non-increasing on $(0,\infty)$. In particular,
\begin{eqnarray}
\frac{\log\EE\left[F_t\right]}{t} & \le & \lim_{h\to 0}\uparrow \frac{\log\EE\left[F_h\right]}{h} \ = \ \EE[f],
\end{eqnarray}
or equivalently,  $\EE\left[e^{tf}\right]\le e^{t\EE[f]+ct^2}$. Using Chernov's bound, we deduce that for all $a,t\ge 0$,
\begin{eqnarray}
\pi\left(f \ge \EE[f]+a\right) & \le & e^{ct^2-at}.
\end{eqnarray}
The claim now follows by setting $t=\frac{a}{2c}$, so as to minimize the right-hand side.
\end{proof}
\section*{Acknowledgments} We thank Prasad Tetali for useful discussions and for bringing some relevant references to our attention.
\bibliographystyle{plain}
\bibliography{ZRP}

\end{document}